\documentclass[11pt]{amsart}
\usepackage{enumerate}
\usepackage{hyperref}
\usepackage{amsthm,amsfonts,amsmath,amscd,amssymb}
\usepackage{xcolor}
\usepackage{graphicx}
\usepackage{caption}
\usepackage{subcaption}
\let\oldmarginpar\marginpar
\renewcommand\marginpar[1]{\-\oldmarginpar[\raggedleft\footnotesize #1]%
{\raggedright\footnotesize #1}}
\theoremstyle{plain}
\newtheorem{thm}{Theorem}[section]
\newtheorem{lemma}[thm]{Lemma}
\newtheorem{prop}[thm]{Proposition}
\newtheorem{coro}[thm]{Corollary}
\newtheorem{question}[thm]{Question}
\theoremstyle{definition}

\newtheorem{definition}[thm]{Definition}

\theoremstyle{remark}

\numberwithin{equation}{section}

\newcommand{\R}{\mathbb{R}}

\newcommand{\Z}{\mathbb{Z}}

\newcommand{\dd}{\partial}

\newcommand{\es}{\varnothing}
\newcommand{\op}{\operatorname}

\newcommand{\sse}{\subseteq}

\newcommand{\x}{\times}
\newcommand{\sm}{\setminus}
\newcommand{\pt}{\text{point}}

\newcommand{\wt}{\widetilde}

\newcommand{\ol}{\overline}

\DeclareMathOperator{\Bisec}{Bisec}

\begin{document}
\title{Amenable groups and smooth topology of $4$-manifolds}
\author{Michael Freedman \and Larry Guth\and Emmy Murphy}

\begin{abstract}
A smooth five-dimensional $s$-cobordism becomes a smooth product if stabilized by a finite number $n$ of $(S^2\times S^2\times [0,1])$'s.  We show that for amenable fundamental groups, the minimal $n$ is subextensive in covers, i.e., $\frac{n(\text{cover})}{\text{index cover}} \rightarrow 0$. We focus on the notion of sweepout width, which is a bridge between $4$-dimensional topology and coarse geometry.
\end{abstract}

\maketitle

\section{Introduction}\label{sec: intro}

Let $W$ is an $s$-cobordism between smooth manifolds $N$ and $N'$. Since pioneering work in the 1980's it has been known that $N$ and $N'$ are not necessarily diffeomorphic \cite{D}, but they are homeomorphic \cite{F, FQ} when the fundamental group is of a certain type, called ``good'' (a kind of time-dependent definition) by practitioners. Good groups \cite{F,FQ,FT,KQ} were built from the class of polynomial growth groups and later the more general subexponential class by closure under four operations: extension, quotient, direct limit, and subgroup. It is tempting to guess that good groups can be extended to all amenable groups, but this is not presently known.

Though the question of any classification up to diffeomorphism seems completely intractable at this point, it was noticed by Quinn \cite{Q} in the 1970's that these subtleties disappear under stabilization. That is, there exists a natural number $n$ so that $(W_n; N_n, N'_n)$ is a product cobordism, where for a $4$-manifold $N$ (or $N'$) $N_n := N \# (S^2 \x S^2)^{\#n}$ and for a $5$-dimensional cobordism $W$, $W_n$ denotes the ``connected sum $\x [0,1]$'' with $(S^2 \x S^2)^{\#n}$ summed parametrically along a vertical arc in $W$. For the remainder of the paper we denote $S^2\x S^2$ by $X$, and ``connected sum $\x [0,1]$'' by $\ol\#$.

This paper is concerned with how small $n$ can be. When $N$ and $N'$ are simply connected complete intersections, it follows from \cite{M} (and a little five-dimensional surgery) that $n = 1$ suffices.  Beyond this, little seems to be known: no argument from gauge theory or any other direction excludes (as far as we know) the possibility that $n = 1$ suffices universally.

Suppose $G = \pi_1 W$ and $\{H_i\}$ is an infinite collection of finite index subgroups of $G$ with $H_0 = G$. We let $I_i$ denote the index $[H_i : G]$.  Consider the corresponding covering spaces $W^i$ with $\pi_1 W^i = H_i$ and define $n_i$ to be the minimal $n$ so that $(W^i)_{n_i}$ is a product. It is clear that $n_i \leq n_0I_i$, since $(W_n)^i = (W^i)_{nI_i}$, i.e. the covering space of $W \ol{\#} X^{\#n}$ corresponding to $H_i$ is $W^i\ol{\#}X^{\#(nI_i)}$. The main theorem of this paper is:

\begin{thm} \label{thm: main}
If $\pi_1W$ is an amenable group, then for any sequence of subgroups $H_i \sse \pi_1W$ with $I_i \to \infty$ we have $$\lim \limits_{i\to \infty}\frac{n_i}{I_i} = 0.$$
\end{thm}
More generally the theorem holds if the maximal residually finite quotient of $\pi_1W$ is amenable. Recall that the maximal residually finite quotient of a group $G$ is the quotient group $G / (\bigcap_i H_i)$, where the intersection ranges over all finite index subgroups.

The main theorem is actually a combination of two theorems, one in smooth topology and one in coarse geometry. Before stating these we discuss the notion of sweepout width of a coset space.

We consider a finitely generated group $G$ as a discrete metric space by choosing a finite generating set and building its Cayley graph. The distance between two group elements is then defined to be the minimal number of edges needed to join them in the Cayley graph. Let $H \sse G$ be a finite index subgroup, and let $C = G/H$ be the space of cosets, with the induced metric. Given a set $A \sse C$, define $\dd A$ to be the set of all vertices in $A$ to points in the complement of $A$. A \emph{sweep out} $\mathfrak{F}$ of $C$ is a sequence of subsets $F_j \sse C$, $0\leq j\leq |C|$, $F_j\sse F_{j+1}$, with $|F_j| = j$. Define the \emph{width} of $\mathfrak{F}$ by $\op{w}(\mathfrak{F}) := \max_{j = 0,\ldots,k} \left|\partial F_j\right|$.

\begin{definition}
We say that $\op{sw}(C) := \min_{\mathfrak{F}}\op{w}(\mathfrak{F})$ is the \emph{sweepout width of $C$}.
\end{definition}

Theorem \ref{thm: main} follows immediately from the following two theorems.

\begin{thm}\label{thm: main amenable}
Let $G$ be an amenable group, let $H_i$ be any sequence of finite index subgroups with $I_i \to \infty$, and let $C_i = G/H_i$. Then $$\lim\limits_{i \to \infty}\frac{\op{sw}(C_i)}{I_i} = 0.$$ 
\end{thm}

In fact, the conclusion of the theorem holds exactly when the maximally residually finite quotient of $G$ is amenable.

\begin{thm}\label{thm: main topology}
Let $(W, N, N')$ be a smooth $5$-dimensional $s$-cobordism, let $H_i \sse \pi_1W$ be a sequence of finite index subgroups and let $W^i$ be the covering space of $W$ corresponding to $H_i = \pi_1W^i$ (thus $W^i$ is an $s$-cobordism between the corresponding covering spaces of $N$ and $N'$). Let $C_i = \pi_1W / H_i$ and let $n_i$ be the minimal number so that $W^i \ol{\#} X^{\#n_i}$ is a trivial product. Then there is a constant $K$ not depending on $i$ so that $n_i \leq K \op{sw}(C_i)$.
\end{thm}

Often Theorem \ref{thm: main topology} gives better bounds than Theorem \ref{thm: main}, if we have additional knowledge of the group $\pi_1W$ or the sequence $H_i$. For example:

\begin{coro}\label{cor: pi1 Z}
Let $(W, N, N')$ be a smooth $5$-dimensional $s$-cobordism as above, suppose $\pi_1W \cong \Z$, and let $W^i$ be the covering space of $W$ corresponding to the subgroup $i\Z \sse \Z$. Then there is a constant $K$ so that for all $i$ $W^i \ol{\#} X^{\# K}$ is a trivial product.
\end{coro}

\begin{proof}
Define $\mathfrak{F}^i = \{F^i_j\}_{j=1}^i$ by $F^i_j = \{1, \ldots j\}$. Then $\mathfrak{F}^i$ is a scale-$1$ sweepout of $\Z / i\Z$ with width $\op{w}(\mathfrak{F}^i) = 2$.
\end{proof}

Suppose $N$ and $N'$ are closed $4$-manifolds and there is a homeomorphism $f: N \to N'$ with vanishing Kirby-Siebenmann invariant. Then $f \x \op{id}: N \x (0,1) \to N' \x (0,1)$ has a controlled isotopy to a diffeomorphism, and using this diffeomorphism one can build a smooth $s$-cobordism $W$ with $\dd W = N - N'$. Therefore if $\pi_1 N$ is amenable Theorem \ref{thm: main} tells us that the number of $S^2 \x S^2$ connect sums required to make $N$ diffeomorphic to $N'$ is subextensive in covers.

\begin{question}
Does the same hold true if $\op{ks}(f) = 1$?
\end{question}

Gompf \cite{Gom} and Kreck \cite{K} independently proved that if $\op{ks}(f) = 1$ then there is a homeomorphism $\wt f: N \# S^2 \x S^2 \to N' \# S^2 \x S^2$ so that $\op{ks}(\wt f) = 0$. But even a single copy of $S^2 \x S^2$ connect summed before construction of $W$ leads to an extensive number of $S^2 \x S^2$'s in covers, so this fact cannot be directly applied.

The remainder of the paper is divided into two sections. Theorem \ref{thm: main amenable} is proved in Section \ref{sec: coarse} and Theorem \ref{thm: main topology} is proved in Section \ref{sec: smooth}.

The authors are grateful to Pierre Pansu, Richard Schwartz, and Romain Tessera for conversations about polynomial growth groups which was very useful for developing intuition.

\section{Coarse geometry} \label{sec: coarse}

\subsection{Folner profiles and pushforwards}

Suppose $\Gamma$ is a graph.  We will also write $\Gamma$ for the set of vertices in $\Gamma$ and $E(\Gamma)$ for the set of edges in $\Gamma$.  In this section, if $A \sse \Gamma$ is a set of vertices, we define $\partial A$ to be set of edges from $A$ to the complement of $A$. (This is slightly different from the definition from the previous section, where we defined $\dd A$ to be the set of vertices in $A$ which are adjacent to the complement of $A$. Since we always work with graphs with uniformly bounded valance, $|\dd A|$ is the same with either definition, up to a constant.)

The Folner profile of $\Gamma$ records some information about the sizes of sets and their boundaries in $\Gamma$.  We define the Folner profile of $\Gamma$, $\Phi_\Gamma$ by

$$ \Phi_{\Gamma} (V) := \min_{A \sse \Gamma, |A| \le V} \frac{ | \partial A|} {|A|}. $$

We note that $\Phi_{\Gamma}(V)$ is non-increasing.

We will be mainly interested in Cayley graphs of groups and their quotients. Suppose that $G$ is a finitely generated group with generating set $S$.  (We assume that $S$ is closed under taking inverses: $s \in S$ if and only if $s^{-1} \in S$.) We denote the Cayley graph of $G$ by $\Gamma(G, S)$. Note that $G$ is amenable if and only if

$$ \lim_{V \rightarrow \infty} \Phi_{\Gamma(G,S)} (V) = 0. $$

We often denote $\Phi_G := \Phi_{\Gamma(G,S)}$ when the generating set if clear from context.

We are interested in coset spaces of $G$.  Suppose that $H \sse G$ is a subgroup of $G$.  Given a generating set $S$ for $G$, there is a natural graph structure on the quotient $G / H$ called the \emph{Schreier graph}, which we denote $\Gamma( G/H, S)$.  The set $ G / H$ is the set of right cosets $g H \sse G$.  There is an edge between $g_1 H$ and $g_2 H$ if and only if $s g_1 H = g_2 H$ for some $s \in S$, if and only if $s g_1 \in g_2 H$ for some $s \in S$.  Here is an equivalent description of $\Gamma( G/H, S)$.  The group $H$ acts on $G$ on the right, and so $H$ acts by isomorphisms on the graph $\Gamma(G, S)$.  The graph $\Gamma (G/H, S)$ is the quotient graph for this action. Again, if the generating set $S$ is clear from the context, we abbreviate $\Phi_{G/H} := \Phi_{\Gamma(G/H,S)}$.

\begin{prop} \label{prop: folquot} If $G$ is a finitely generated group with generating set $S$ and $H \sse G$ is a subgroup, then for all $V \ge 1$, 

$$ \Phi_{\Gamma(G/H, S)} (V) \le \Phi_{\Gamma (G,S)} (V).  $$
\end{prop}

We briefly discuss why this Proposition is tricky and the main idea of the proof.  Suppose that $A \sse G$ with $|A| \le V$, and $\frac{ |\partial A|}{|A|} = \Phi_G(V)$.  We want to find a set $A' \sse G/H$ with $|A'| \le V$ and $\frac{ |\partial A'|}{|A'|} \le \Phi_G(V)$.  Let $\pi: G \rightarrow G/H$ be the quotient map, the most obvious set to look at is $\pi(A)$. Clearly $|\pi(A)| \le |A|$, but it seems hard to control the ratio $\frac{ |\partial \pi(A)|}{|\pi(A)|}$.  Part of the problem is that some points in $\pi(A)$ may have many preimages in $A$ while other points in $\pi(A)$ may have few preimages in $A$. Instead of working just with $\pi(A)$, it turns out to be more natural to consider a different type of pushforward from $G$ to $G/H$ which takes into account the number of preimages.

Suppose that $f: G \rightarrow \R$ is a function with finite support.  Then we define the pushforward $\pi_* f: G/H \rightarrow \R$ by

$$ \pi_* f (x) = \sum_{g \in \pi^{-1}(x)} f(x). $$

To make good use of the pushforward for functions, we recall the notion of the gradient of a function on a graph.  If $\Gamma$ is a graph, $f: \Gamma \rightarrow \R$, and $e$ is an oriented edge of $\Gamma$ from $v_1$ to $v_2$, then $\nabla f(e) := f(v_2) - f(v_1)$.  If $e$ is an unoriented edge with boundary vertices $v_1$ and $v_2$, then we can still define $| \nabla f| (e)$:

$$ | \nabla f | (e) := | f(v_2) - f(v_1) |. $$

This gradient is closely related to the size of the boundary of a set: for any set $A \sse \Gamma$, we have

$$ | \partial A| = \sum_{e \in E(\Gamma)} |\nabla \chi_A | (e). $$

The set $A'$ that we construct comes from $\pi_* \chi_A$, where $\chi_A$ denotes the characteristic function of $A$.  More precisely, $A'$ is a super-level set of the pushforward, $A' = \{ \pi_* \chi_A \ge \lambda \}$ for some $\lambda$. The first step in the proof of Proposition \ref{prop: folquot} is

\begin{lemma} \label{lem: pushdown} If $f: G \rightarrow \R$ has finite support, then

$$ \sum_{e \in E(G/H)} | \nabla \pi_* f |(e)  \le \sum_{e \in E(G)} |\nabla f|(e) . $$

\end{lemma}

\begin{proof} For each $x \in G/H$ and each $s \in S$, either $s x = x$ or there is an edge from $x$ to $sx$ in the graph $\Gamma(G/H, S)$.  When we sum over all pairs $s, x$, we get

$$ 2 \sum_{e \in E(G/H)} | \nabla \pi_* f |(e) = \sum_{s \in S} \sum_{x \in G/H} | \pi_* f(sx) - \pi_*f(x) |. $$

\noindent (Each edge $e \in E(G/H)$ is counted twice on the right-hand side because it has two endpoints.  On the other hand, if $sx = x$, then $ | \pi_* f(sx) - \pi_*f(x) | = 0$ and so it doesn't contribute anything.)  The right-hand side is 

$$ = \sum_{s \in S} \sum_{x \in G/H} \left| \sum_{g \in \pi^{-1}(x)} f(sg) - f(g) \right| \le \sum_{s \in S} \sum_{g \in G} \left| f(sg) - f(g) \right| = 2 \sum_{e \in E(G)} |\nabla f| (e). $$

\end{proof}

\begin{proof}[Proof of Proposition \ref{prop: folquot}.] Pick $A \sse G$ with $|A| \le V$ and with $\frac{ |\partial A|}{|A|} = \Phi_G(V).$

We note that $\sum_{g \in G} \chi_A(g) = |A|$, and so 

$$\sum_{x \in G/H} \pi_* \chi_A(x) = |A|.$$  

We also note that $\sum_{e \in E(G)} | \nabla \chi_A | = | \partial A|$, and so by Lemma \ref{lem: pushdown}

$$ \sum_{e \in E(G/H)} | \nabla \pi_* \chi_A | (e) \le |\partial A|. $$

Next we rewrite these two equations in terms of the super-level sets $\{ \pi_* \chi_A \ge \lambda \}$, for integers $\lambda$.  The first equation becomes: 

$$ \sum_{ \lambda \ge 1} | \{ \pi_* \chi_A \ge \lambda \} | = | A| . $$

The second equation becomes:

$$ \sum_{ \lambda \ge 1} | \partial \{ \pi_* \chi_A \ge \lambda \} | = \sum_{e \in E(G/H)} | \nabla \pi_* \chi_A| (e) \le |\partial A|. $$

Therefore, there exists some integer $\lambda \ge 1$ so that

$$ | \partial \{ \pi_* \chi_A \ge \lambda \} | \le \frac{ |\partial A|}{|A|}  | \{ \pi_* \chi_A \ge \lambda \} |. $$

We let $A' :=  \{ \pi_* \chi_A \ge \lambda \}$.  We note that $A' \sse \pi(A)$, and so $|A'| \le |A| \le V$.  Therefore we have shown

$$ | \partial A' | \le \frac{ |\partial A|}{|A|}  |A'| = \Phi_G(V) |A'|. $$

\end{proof}

\subsection{Approximate bisections}

We use the Folner profile to control the size of approximate bisections in $G/H$.  If $A$ is a subset of a graph $\Gamma$, we consider writing $A$ as a disjoint union $B_1 \sqcup B_2$, where each of the sets $B_i$ obeys $(1/4) |A| \le |B_i| \le (3/4) |A|$.  We call such a decomposition an approximate bisection of $A$.  Among all the approximate bisections, we are interested in minimizing the number of edges between $B_1$ and $B_2$.  

$$\Bisec(A) := \min_{A = B_1 \sqcup B_2; (1/4) |A| \le |B_i| \le (3/4) |A|} | E(B_1, B_2) |. $$

\noindent We write $E(B_1, B_2)$ for the set of edges from $B_1$ to $B_2$.

\begin{prop} \label{folbisec} Suppose that $G$ is a finitely generated group and $H \sse G$ is a cofinite subgroup.  If $A \sse G/H$, and $|A| \ge 4$, then

$$ \Bisec(A) \le \Phi_{G/H} ( |A|/4) |A| \le \Phi_G ( |A|/4) |A| . $$

\end{prop}

In particular, if $G$ is an amenable group, then $\Phi_G(V) = o(V)$ and so $\Bisec(A) = o( |A| )$ for any subset $A$ of any finite quotient $G/H$.  Informally, if $G$ is amenable, any large subset of any finite quotient can be bisected efficiently.

We start by proving the following weaker lemma:

\begin{lemma} \label{folbite} Suppose that $G$ is a finitely generated group and $H \sse G$ is a cofinite subgroup.  If $A \sse G/H$, with $|A| \ge 2$, then there exists a subset $A' \sse A$ with $|A'| \le (1/2) |A|$ so that 

$$ | E(A', A \setminus A') | \le  \Phi_{G/H}( |A| / 2) |A'|. $$

\end{lemma}

If $\Phi_{G/H}$ is small, then this lemma says that we can cut off a piece $A' \sse A$ in an efficient way -- efficient in the sense that the number of edges we need to cut to separate $A'$ from the rest of $A$ is much smaller than the number of vertices in $A'$.

\begin{proof} Let $F \sse G/H$ be a finite set that realizes $\Phi_{G/H}( |A| / 2)$.  In other words, $|F| \le |A| / 2$, and

$$ \frac{ | \partial F |}{|F|} = \Phi_{G/H} ( |A| / 2 ). $$

We consider the translations of $F$ by the left action of $G$ on $G/H$.  The set $A'$ will be $g F \cap A$ for a group element $g$ chosen by an averaging argument.  We cannot exactly average over $G$ because $G$ is infinite.  However, we have assumed that $G/H$ is finite, and so the elements of $G$ act on $G/H$ in only finitely many ways.
Consider the homomorphism $G \rightarrow Isom(G/H)$, sending each element $g \in G$ to the corresponding isomorphism of the graph $G/H$.  The image of this map is a finite subgroup $Q \sse Isom(G/H)$.  The group $G$ acts transitively on $G/H$, and so $Q$ acts transitively on $G/H$.  Therefore, we see that for any $x_1, x_2 \in G/H$,

\begin{equation} \label{actpt} | \{ q \in Q | q x_1 = x_2 \} | = |Q| / |G/H|. \end{equation}

With this fact we can get good estimates about the average behavior of $q F$.  First we compute $\sum_{q \in Q} |q F \cap A|$.  Given a point $x \in qF \cap A$ for some $q$, we associate the point $x' = q^{-1}(x) \in F$ and the point $x \in A$.  Now for any points $x' \in F$ and $x \in A$, there are $|Q| / |G/H|$ choices of $q \in Q$ so that $q(x') = x$.  There are $|F|$ choices of $x'$ and $|A|$ choices of $x$.  Therefore, we see that

$$ \sum_{q \in Q} |q F \cap A| = |F| |A| |Q| / |G/H|. $$

Next we would like to upper bound $ \sum_{q \in Q} | E(q F \cap A, A \setminus qF) |$.  If $e$ is an edge from $x_1 \in qF \cap A $ to $x_2 \in A \setminus qF$, then we see that $q^{-1}(x_1) =x_1' \in F$, $q^{-1}(x_2) =x_2' \notin F$, and so $q^{-1}(e) = e' \in \partial F$.  We associate to an edge $e \in E(q F \cap A, A \setminus qF)$ the point $x_1 \in A$ and the edge $e' \in \partial F$.  Given an edge $e' \in \partial F$ from $x_1' \in F$ to $x_2' \notin F$, and an element $x_1 \in A$, there are $|Q| / |G/H|$ elements $q \in Q$ taking $x_1'$ to $x_1$.  Therefore,

$$ \sum_{q \in Q} | E(q F \cap A, A \setminus qF) | \le |\partial F| |A| |Q| / |G/H|. $$

Therefore, there exists some $q \in Q$ so that

$$ | E(q F \cap A, A \setminus qF) | \le \frac{ |\partial F|} {|F|}   |q F \cap A|. $$

We set $A' = q F \cap A$.  The last equation tells us that

$$ | E(A' , A \setminus A')| \le \Phi_{G/H}( |A| / 2) | A'|. $$

We also know that $|A'| \le |F| \le |A|/2$.  \end{proof}

We now find an efficient bisection of $A$ by using Lemma \ref{folbite} repeatedly.  We keep efficiently cutting off small pieces of $A$ until we have approximately bisected $A$.

\begin{proof}[Proof of Proposition \ref{folbisec}]
Let $A \sse G/H$.  Let $A' \sse A$ be as in Lemma \ref{folbite}.  Lemma \ref{folbite} tells us that

$$ |E(A', A \setminus A') | \le  \Phi_{G/H}( |A| / 2) | A'| \le  \Phi_{G/H}( |A| / 4) | A'|.  $$

\noindent We know that $|A'| \le (1/2) |A|$.  If $|A'| \ge (1/4) |A|$, then we set $B_1 = A'$ and $B_2 = A \setminus A'$, and we have the desired bisection of $A$.  If $|A'| \le (1/4) |A|$, then we set $A_2 = A \setminus A'$, and we apply Lemma \ref{folbite} to $A_2$.  (We also let $A_1 = A$ and $A_1' = A'$.)  Lemma \ref{folbite} tells us that there is a subset $A_2' \sse A_2$ with $|A_2'| \le (1/2) |A_2|$, and

$$ |E(A_2', A_2 \setminus A_2') | \le  \Phi_{G/H}( |A_2| / 2) | A_2'| \le  \Phi_{G/H}( |A| / 4) | A_2'|.  $$

\noindent (The last step came from noting that $|A_2| \ge (1/2) |A|$.)  We repeat this process until the first step $s$ where $|A_s| \le (3/4) |A|$.  For $j \le s-1$, we see that $|A_j| \ge (3/4) |A|$, and so

$$ |E(A_j', A_j \setminus A_j') | \le  \Phi_{G/H}( |A_j| / 2) | A_2'| \le  \Phi_{G/H}( |A| / 4) | A_j'|.  $$

Then we set $B_1 = A_1' \cup ... \cup A_{s-1}'$, and $B_2 = A \setminus B_1$.  We first check that $(1/4) |A| \le |B_2| \le (3/4) |A|$.  Note that $|B_2| = |A_s| \le (3/4) |A|$.  On the other hand, $|A_{s-1}| \ge (3/4) |A|$, and $|A_{s-1}'| \le (1/2) |A_{s-1}| \le (1/2) |A|$, and so $|B_2| = |A_{s-1} \setminus A_{s-1}'| \ge (1/4) |A|$.  Next we estimate $E(B_1, B_2)$.

$$ | E(B_1, B_2) | \le \sum_{j=1}^{s-1} | E(A_j', A_j \setminus A_j') | \le \Phi_{G/H}( |A| / 4) \sum_{j=1}^{s-1} |A_j'| =  \Phi_{G/H}( |A| / 4) |B_1|. $$

\end{proof}

\subsection{Sweepout estimates}

Suppose that $A$ is a finite subset of a graph $\Gamma$.  A sweepout $\mathfrak{F}$ of $A$ is a nested sequence of subsets

$$ \emptyset = F_0 \sse F_1 \sse \ldots \sse S_{|A|} = A, $$

\noindent with $|F_j| = j$.  The width of a sweepout is defined by

$$ \op{w}(\mathfrak{F}) := \max_j | E (F_j, A \setminus F_j) |. $$

In the special case of most interest to us, the set $A$ may be the entire graph $\Gamma$, in which case $\op{w}(\mathfrak{F}) = \max_j | \partial F_j |$.

The sweepout width of a set $A$ is the minimal width of any sweepout of $A$:

$$ \op{sw} (A) := \min_{\mathfrak{F} \textrm{ a sweepout of } A} \op{w}(\mathfrak{F}). $$

In \cite{BS}, Balacheff and Sabourau introduced an inductive procedure to construct sweepouts from bisections.  They worked on Riemannian manifolds instead of graphs (which is harder!), and in the case of graphs their ideas easily imply the following lemma.

\begin{lemma} \label{lembalsab} Suppose that $A$ is a finite subset of a graph $\Gamma$.  Then there exists a subset $B \sse A$ with $|B| \le (3/4) |A|$, so that

$$ \op{sw} (A) \le \Bisec (A) + \op{sw}(B). $$

\end{lemma}

\begin{proof} By the definition of $\Bisec(A)$, we can decompose $A$ into two disjoint pieces, $B_1$ and $B_2$, so that 
$(1/4) |A| \le |B_i| \le (3/4) |A|$ and

$$ |E (B_1, B_2) | \le \Bisec(A). $$

Let $\mathfrak{F}^1 = \{F^1_j\}$ be a sweepout of $B_1$ and let $\mathfrak{F}^2$ be a sweepout of $B_2$.  We will assemble $\mathfrak{F}^1$ and $\mathfrak{F}^2$ to get a sweepout $\mathfrak{F}$ of $A$ obeying

$$\op{w}(\mathfrak{F}) \le \Bisec (A) + \max ( \op{w}(\mathfrak{F}^1), \op{w}(\mathfrak{F}^2) ). $$

Then choosing $\mathfrak{F}^1$ and $\mathfrak{F}^2$ so that $\op{w}(\mathfrak{F}^1) = \op{sw}(B_1)$ and $\op{w}(\mathfrak{F}^2) = \op{sw}(B_2)$ gives the conclusion.

We let $F_j = F^1_j$ for $0 \le j \le |B_1|$.  Then we let $F_j = B_1 \cup F^2_{j - |B_1|}$ for $|B_1| \le j \le |A|$.  It is straightforward to check that $\mathfrak{F} = \{F_j\}$ is a sweepout of $A$, and we just have to estimate $|E( F_j, A \setminus F_j)|$.  
If $j \le |B_1|$ we have that $F_j = F^1_j \sse B_1$, and so 

$$|E( F_j, A \setminus F_j)| \le |E(F^1_j, B_1 \setminus F^1_j)| + | E( B_1, B_2) | \le \op{w}(\mathfrak{F}^1) + \Bisec(A). $$

If $j \ge |B_1|$ we have that $F_j = B_1 \cup F^2_{j'}$ for $j' =  j - |B_1|$, and so

$$|E( F_j, A \setminus F_j)| \le |E(B_1, B_2)| + |E(F^2_{j'}, B_2 \setminus F^2_{j'})| \le \Bisec(A) + \op{w}(\mathfrak{F}^2). $$

\end{proof}

Using the bounds for $\Bisec(A)$ from Proposition \ref{folbisec}, and iterating Lemma \ref{lembalsab}, we get the following estimate which implies Theorem \ref{thm: main amenable}.

\begin{thm} \label{folsw} Suppose that $G$ is a finitely generated group with generating set $S$, and $H \subset G$ is a cofinite subgroup.  For any subset $A \sse G/H$,

\begin{equation} \label{folswest}
\op{sw}(A) \le 6 + 2 \sum_{k=1}^\infty \Phi_G \left( \frac{1}{4} \cdot \left(\frac{3}{4} \right)^k \cdot |A| \right) \left(\frac{3}{4} \right)^k |A|. 
\end{equation}

\end{thm}

\begin{proof} Iterating Lemma \ref{lembalsab}, we see that there is a sequence of subsets

$$A = A_0 \supseteq A_1 \supseteq A_2 \supseteq \ldots \supseteq A_{\text{final}}$$

so that $|A_{j+1}| \le (3/4) |A_j|$, $|A_{\text{final}}| \le 4$ and 

$$ \op{sw}(A) \le \sum_j \Bisec(A_j) + \op{sw}(A_{\text{final}}) \le 6 + \sum_j \Bisec(A_j). $$

Applying Proposition \ref{folbisec} to estimate $\Bisec(A_j)$, we see that 

$$ \op{sw}(A) \le 6 + \sum_j \Phi_G ( |A_j|/4) |A_j|. $$

Clearly $|A_j| \le |A|$.  Also, there is at most one $j$ for which $|A_j|$ lies in any interval $I_k$ of the form

$$ I_k := \left( \left(\frac{3}{4} \right)^{k} |A| , \left(\frac{3}{4} \right)^{k-1} |A| \right]. $$

If $|A_j| \in I_k$, then, since $\Phi_G$ is non-increasing,

$$ \Phi_G ( |A_j|/4) |A_j| \le 2 \Phi_G \left(  \frac{1}{4} \cdot \left(\frac{3}{4} \right)^{k} |A| \right)  \left(\frac{3}{4} \right)^{k} |A|. $$

Therefore, 

$$ \op{sw}(A) \le 6 + \sum_j  \Phi_G ( |A_j|/4) |A_j| \le 6 + 2 \sum_{k=1}^\infty \Phi_G \left( \frac{1}{4} \cdot \left(\frac{3}{4} \right)^k \cdot |A| \right) \left(\frac{3}{4} \right)^k |A|. $$

\end{proof}

To complete the proof of Theorem \ref{thm: main amenable}, we just take $A = G/H_i$ and note the above bound implies $$\lim\limits_{|G/H_i| \to \infty} \frac{\op{sw}(G/H_i)}{|G/H_i|} = 0$$ as long as $\lim_{V \to \infty}\Phi_G(V) = 0$. The following bounds may also be useful in practice.

\begin{coro}
Suppose $G$ is a finitely generated group, and let $H_i$ be any sequence of finite index subgroups with $\lim |G/H_i| = \infty$. If $\Phi_G(V) \lesssim V^{-\alpha}$ for some $\alpha > -1$, then
$\op{sw}(G/H_i) \lesssim |G/H_i|^{1 - \alpha}$. If $\Phi_G(V) \lesssim (\log V)^{-1}$, then
$\op{sw}(G/H_i) \lesssim |G/H_i| (\log |G/H_i|)^{-1}$.
\end{coro}

\begin{proof} Plug the given bound on $\Phi_G$ into Equation \ref{folswest} and then note that the series decays exponentially, so it is dominated by its first term.
\end{proof}

\section{Smooth topology} \label{sec: smooth}

The first step in our discussion of $s$-cobordism theory is to review ideas of Smale, Casson, and Quinn.

$W$ may be presented as $N\times I\;\cup\; 2\text{-handles }\cup\; 3\text{-handles}$.  The all-important middle level $M$ is the upper boundary of $N\times I\;\cup\; 2$-handles.  Inside $M$ there are two collections of disjointly (normally framed) embedded $2$-spheres the $(\text{red},R)$ ascending spheres of the $2$-handles and the $(\text{blue},B)$ descending spheres of the $3$-handles.  Because $W$ is an $s$-cobordism, the intersections are algebraically:
\begin{equation}\label{eq:alg}
\langle R_i,B_j\rangle= \delta_{ij}
\end{equation}
over the group ring $Z[\pi_1(M)]$. This is where Smale took us for $\pi_1 W = \{0\}$, and for general $\pi_1 W$ the result was adapted by Mazur, Stallings, and Barden (independently) giving the $s$-cobordism theorem for $\dim W \geq 6$. For $\dim W = 5$ Casson showed that at the expense of increasing the \emph{geometric} number of intersections $R\cap B$, we could arrange:
\[\pi_1(M\setminus R\cup B) \to \pi_1(M) =: G\]
is an isomorphism.  It is said ``$R\cup B$ is $\pi_1$-negligible.''  Now, correctly framed, disjointly immersed Whitney disks $D$ can be added into the middle level $M$ so that $D\cap (R\cup B)$ consist only of the Whitney circles which pair the excess geometric double points, as in Figure \ref{fig:1}.
\begin{figure}[hbpt]
\begin{center}\includegraphics[scale=0.7]{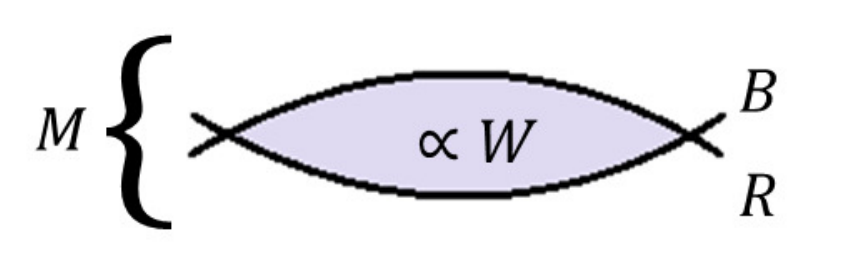}\end{center}
\caption{}
\label{fig:1}
\end{figure}
Finally, Quinn observed that if there were $n_0$ Whitney disk making up $W$ and if there happened to be $n_0$ disjoint copies: $\coprod_{n_0}(S^2\times S^2\setminus \text{pt})\sse M\setminus (B\cup R\cup D)$, then each Whitney disk could be replaced by a (disjointly) embedded one, $D\rightarrow D^\prime$, by tubing the $j^\text{th}$ Whitney disk to the first $S^2$ factor of the $j^\text{th}$ $X$-point and then ``piping off'' (the Norman trick \cite{N}) its intersection to the copies of the second factor sphere.

\begin{figure}[hbpt]
\begin{center}
  \begin{subfigure}{0.3\textwidth}
    \includegraphics[scale=0.7]{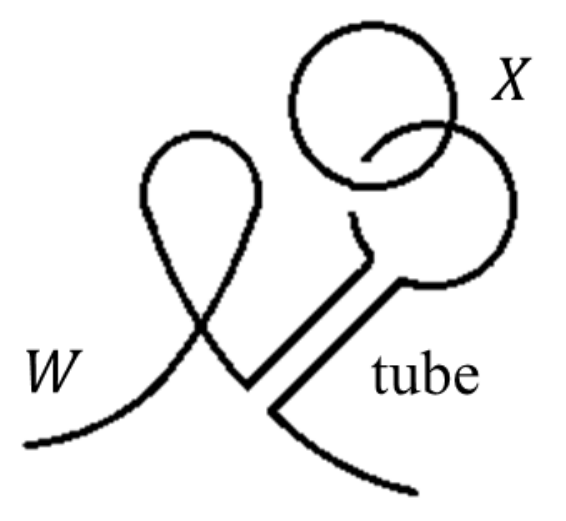}\caption{}
  \end{subfigure}
  \begin{subfigure}{0.3\textwidth}
    \includegraphics[scale=0.7]{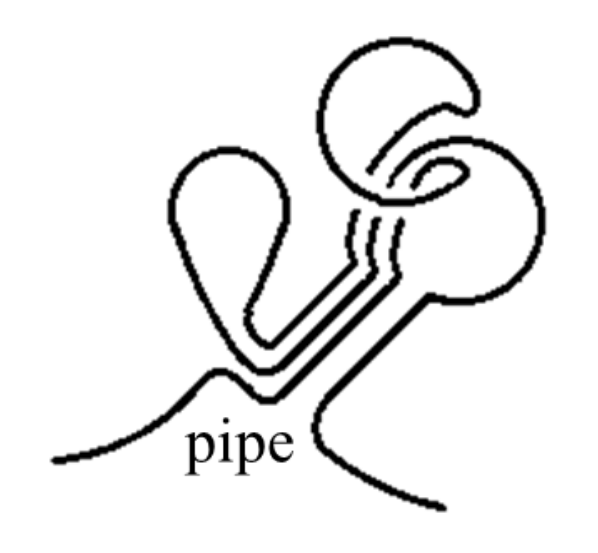}\caption{}
  \end{subfigure}
\end{center}
\caption{}
\end{figure}

Typically $X$ factors will not just ``happen to be'' in $M$, one instead puts them there by forming $W\ol{\#}\left(X^{\#n_0}\times I\right) =: W_{n_0}$.

To relate the topology of $W^i$ to the graph theory of $G/H_i$, we make the following definition. Let $M$ be the middle level of an $s$-cobordism $W$ as above. We define a graph $\Gamma_W$, the \emph{geometric intersection graph} of $W$ as follows. $\Gamma_W$ has a vertex for each red and blue sphere in $M$, and for each geometric intersection of a red and blue sphere we have an edge in $\Gamma_W$.

(In particular, $\Gamma_W$ is bipartite, since all spheres of the same color are disjoint from each other. We could also assign $\pm$ signs to edges based on the sign of the intersections, and note that the ``signed reduction'' of $\Gamma_W$ is the $\delta_{ij}$ graph. However we will not need this additional data, so we ignore it.)

\begin{prop}\label{prop: quasi isom}
Let $W^i$ be a sequence of covering spaces of an $s$-cobordism $W$, corresponding to subgroups $H_i \sse G := \pi_1(W)$. There is a constant $K$ so that
$$\op{sw}(\Gamma_{W^i}) \le K \op{sw}(C_i)$$
where $C_i = G/H_i$.
\end{prop}

\begin{proof}
Let $\wt W$ be the universal cover of $W$. $G$ acts on $\Gamma_{\wt W}$, and the quotient graph $\Gamma_{\wt W} / G = \Gamma_W$, furthermore $\Gamma_{\wt W} / H_i = \Gamma_{W^i}$. For each vertex $v \in \Gamma_W$, choose a lift $v \in \Gamma_{\wt W}$. We can then define maps $f_i:\Gamma_W \x C_i \to \Gamma_{W^i}$ by 
$$f_i(v, gH_i) = (\wt v g)H_i.$$
$f_i$ is a bijection, though clearly it is not an isometry. Let $$c = \max\{\op{dist}_G(g, 1); \text{ there exists } v, w \in \Gamma_W\text{ so that }\op{dist}_{\Gamma_{\wt W}}(\wt v, \wt w g) = 1\}$$ and let $p$ be the cardinality of a generating set of $G$. Let $A \sse C_i$ be a set, and define $$F = f_i(\Gamma_W \x A).$$
We claim $|\dd F| \le p^c |\dd A|$ (here we define $\dd A$ as the subset of vertices in $A$ adjacent to the complement of $A$). Let $\dd_c A$ be the set of all vertices in $A$ which are distance no more than $c$ away from the complement of $A$, then $|\dd_c A| \le p^c |\dd A|$. But by definition of $c$ we see that $\dd F \sse f_i(\Gamma_W \x \dd A)$. We then take $K = |\Gamma_W|(p^c+1)$ to complete the proposition.
\end{proof}

Together with Proposition \ref{prop: quasi isom}, the following proposition completes the proof of Theorem \ref{thm: main topology}.

\begin{prop} \label{prop: resolve via sw}
Let $W^i$ be a sequence of covering spaces of an $s$-cobordism $W$, and let $n_i$ be the minimal integer so that $W^i \ol{\#} (X^{\# n_i}\x I)$ is a trivial product. Then there is a constant $K$ so that $n_i \le K \op{sw}(\Gamma_{W^i})$.
\end{prop}

\begin{proof}

Let $k$ be the maximal number of Whitney disks attached to any one red or blue sphere in $M$. For a set $A \sse \Gamma_{W^i}$, let $\dd A$ be the set of vertices in $A$ which are adjacent to the complement of $A$, and let $\dd_2 A$ be the set of vertices in $A$ which are distance at most $2$ from the complement of $A$. It follows that $|\dd_2 A| \leq k|\dd A|$ since $k$ is the maximal degree of any vertex in $\Gamma_{W^i}$. For each $i$ we choose a sweepout $\mathfrak{F}^i = \{F^i_j\}$ of $\Gamma_{W^i}$ so that $|\dd_2 F^i_j| \le k\op{sw}(\Gamma_{W^i})$. Let $m_i = k^2\op{sw}(\Gamma_{W^i})+k$. We show that $W^i \ol{\#} (X^{\# n_i}\x I)$ is a trivial product, which shows that $K = k^2 + k$ suffices for the conclusion of the proposition.

By convention we think of all red spheres in $M^i$ as fixed, and gradually isotope the blue spheres so that they are in geometrically canceling position, i.e., they intersect only one red sphere, and intersect it exactly once. We trivialize $W^i \ol{\#}(X^{\# m_i} \x I)$ inductively in $j$. More precisely, we prove the following statement: 

Suppose we can find an isotopy of the blue spheres $B \sse M^i \# X^{\# m_i}$ so that all blue spheres in $F^i_j$ are in geometrically canceling position. Suppose further that there are $m_i$ disjoint copies of $X \sm \{\pt\}$ in $M^i \# X^{\# m_i}$, so that each copy intersects only a single blue sphere in $\dd_2 F^i_j$, and each such blue sphere intersects at most $k$ copies of $X \sm \{\pt\}$. Then we can find an isotopy satisfying the same property for $F^i_{j+1}$.

$F^i_{j+1}$ differs from $F^i_j$ by addition of a single vertex. If that vertex is a red sphere, then we do nothing: the blue spheres in $F^i_j$ are exactly the same as the blue spheres in $F^i_{j+1}$, and also the blue spheres in $\dd_2 F^i_j$ are the same as the blue spheres in $F^i_{j+1}$ since $\Gamma_{W^i}$ is bipartite. So suppose $F^i_{j+1} \sm F^i_j$ is a blue sphere, $b \sse M^i$.

After the given isotopy there are at most $k|\dd_2 F^i_j|$ copies of $X$ which intersect some blue sphere, therefore there are at least $m_i -  k|\dd_2 F^i_j| \ge k$ copies of $X$ which are disjoint from all blue spheres. Since $b$ is adjacent to at most $k$ Whitney disks, we can use these $k$ copies of $X$ to replace these Whitney disks with disjointly embedded Whitney disks. We can then apply the Whitney move to these disks, which brings $b$ into algebraically canceling position. The price we pay is that $b$ now intersects these copies of $X$.

It remains to show that any blue spheres which are in $F^i_{j+1} \sm \dd_2 F^i_{j+1}$ can be made disjoint from all copies of $X$. Let $b_0$ be such a blue sphere, and let $r_0$ be its paired red sphere. Since $b_0 \in F^i_{j+1}$ it is in canceling position, meaning it intersects $r_0$ in a single point and intersects no other red spheres. Also $r_0$ intersects no other blue spheres: if $r_0 \cap \wt b \neq \es$, then $\wt b$ is not in geometrically canceling position, which means $\wt b \notin F^i_{j+1}$, contradicting the fact that $b_0 \notin \dd_2 F^i_{j+1}$.

\begin{figure}[hbpt]
\[
  \begin{subfigure}{0.5\textwidth}
  \begin{center}  \includegraphics[scale=0.3]{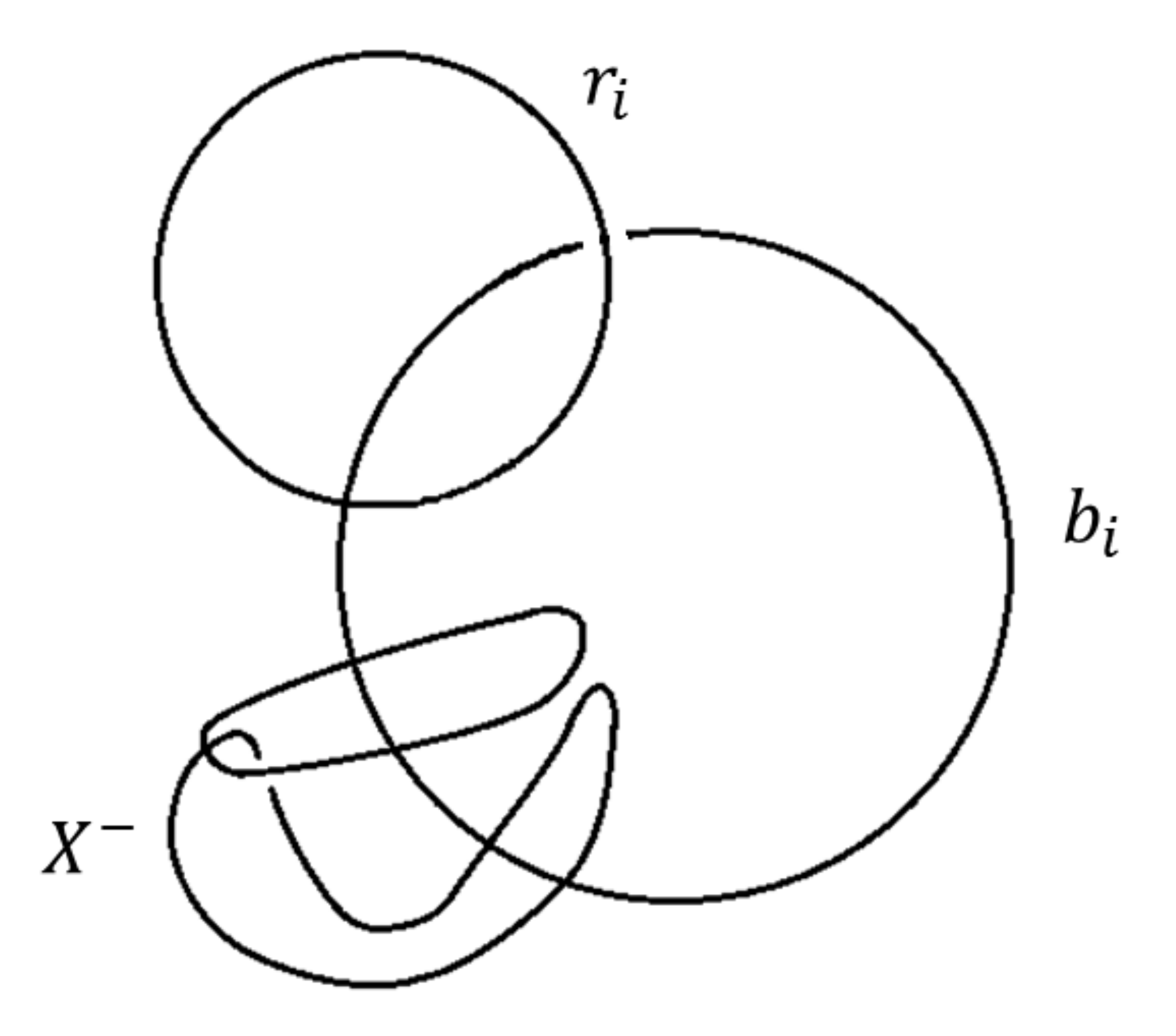}\caption{}
  \end{center}\end{subfigure}
  \begin{subfigure}{0.5\textwidth}
    \includegraphics[scale=0.7]{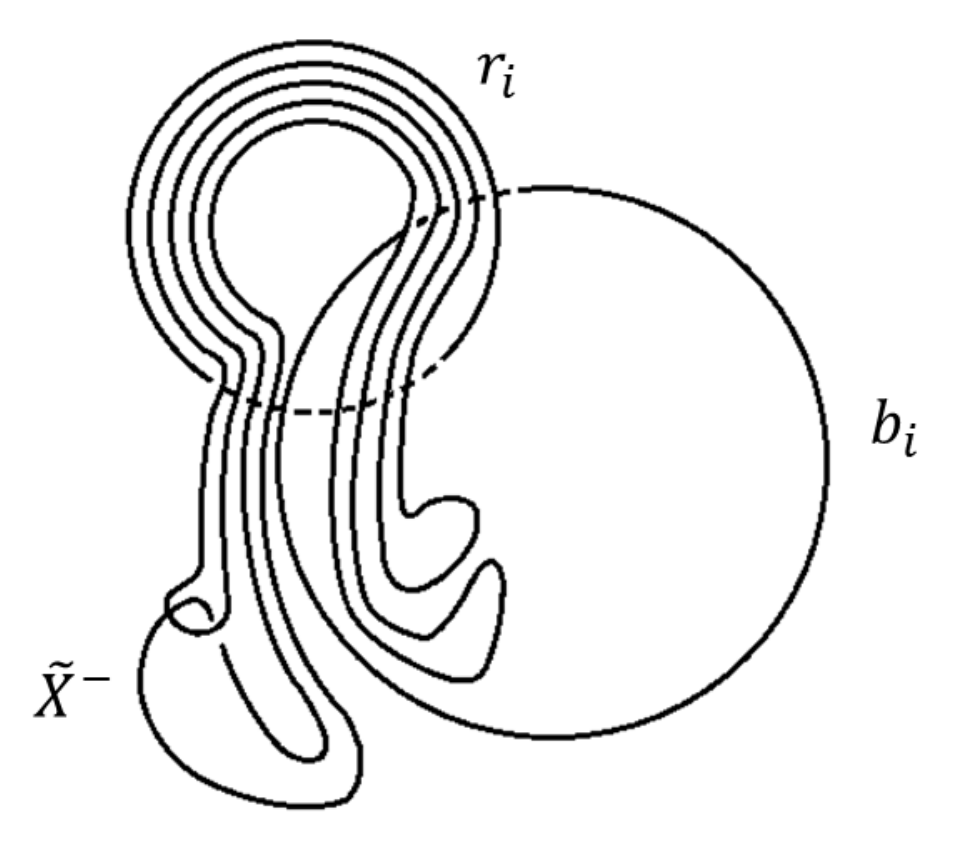}\caption{}
  \end{subfigure}
\]
\caption{}
\label{fig: remove}
\end{figure}

Let $X_0 \sse M^i \# X^{\# m_i}$ be diffeomorphic to $X \sm \{\pt\}$, so that $X_0 \cap b_0 \neq \es$. Let $\Sigma_0 = (\{\pt\}\x S^2) \cup (S^2 \x \{\pt\}) \sse X_0$. Each intersection $\Sigma_0 \cap b_0$ can be resolved by pushing $\Sigma_0$ along an arc in $b_0$ and piping off the intersection to parallel copies of $r_0$. See Figure \ref{fig: remove}. After this process we have constructed a new zero framed copy of $(\{\pt\}\x S^2) \cup (S^2 \x \{\pt\}) \sse M^i \# X^{\# m_i}$, and its neighborhood is a new copy of $X \sm \{\pt\}$ which is disjoint from all blue spheres and all other copies of $X \sm \{\pt\}$. This completes the induction.
\end{proof}

\end{document}